\documentclass[12pt]{article}
\usepackage[a4paper,top=25mm,bottom=25mm,inner=30mm,outer=20mm]{geometry}
\usepackage[svgnames,x11names]{xcolor}
\usepackage{amsfonts,amsmath,amsthm,boldline,colortbl,comment,enumerate,float,graphicx,import,makecell,mathtools,multirow,pdflscape,subcaption,tikz}
\mathtoolsset{centercolon}
\usetikzlibrary{through,intersections,calc}
\usepackage{hyperref}

\newtheorem{theorem}{Theorem}
\newtheorem{lemma}[theorem]{Lemma}

\theoremstyle{definition}

\newtheorem{definition}[theorem]{Definition}

\makeatletter
\newcommand{\gauss}[2]{\genfrac{[}{]}{0pt}{}{#1}{#2}\@ifnextchar\bgroup{_}{}}

\newcommand{\rref}[2]{\hyperref[#2]{{#1}~\ref*{#2}}}

\definecolor{trivial}{RGB}{136,136,136}
\definecolor{chang}{RGB}{51,34,136}
\definecolor{liu}{RGB}{136,204,238}
\definecolor{koolen}{RGB}{68,170,153}
\definecolor{spence}{RGB}{17,119,51}
\definecolor{ramezani}{RGB}{153,153,51}
\definecolor{cioba}{RGB}{221,204,119}
\definecolor{ihringer}{RGB}{170,68,153}
\definecolor{new}{RGB}{136,34,85}
\definecolor{sporadic}{RGB}{204,102,119}

\newcommand{\soften}{70}
\colorlet{trivial}{trivial!\soften}
\colorlet{chang}{chang!\soften}
\colorlet{liu}{liu!\soften}
\colorlet{koolen}{koolen!\soften}
\colorlet{spence}{spence!\soften}
\colorlet{ramezani}{ramezani!\soften}
\colorlet{cioba}{cioba!\soften}
\colorlet{ihringer}{ihringer!\soften}
\colorlet{new}{new!\soften}
\colorlet{sporadic}{sporadic!\soften}

\title{Cospectral mates for generalized \\Johnson and Grassmann graphs}

\author{Aida Abiad\thanks{\texttt{a.abiad.monge@tue.nl},  Department of Mathematics and Computer Science, Eindhoven University of Technology, The Netherlands\\ Department of Mathematics: Analysis, Logic and Discrete Mathematics, Ghent University, Belgium\\ Department of Mathematics and Data Science, Vrije Universiteit Brussel, Belgium}\quad Jozefien D'haeseleer\thanks{\texttt{Jozefien.Dhaeseleer@UGent.be}, Department of Mathematics: Analysis, Logic and Discrete Mathematics, Ghent University, Belgium} \ Willem H. Haemers\thanks{\texttt{haemers@uvt.nl}, Department of Econometrics and Operations Research, Tilburg University, The Netherlands}\quad Robin Simoens\thanks{\texttt{Robin.Simoens@UGent.be}, Department of Mathematics: Analysis, Logic and Discrete Mathematics, Ghent University, Belgium}}

\date{}

\begin{document}

\maketitle

\begin{abstract}

We provide three infinite families of 
graphs in the Johnson and Grassmann schemes that are not uniquely determined by their spectrum. We do so by constructing graphs that are cospectral but non-isomorphic to these graphs.\\

\noindent\textbf{MSC Subject:} 05C50, 15A18\\
\noindent \textbf{Keywords:} Graph, Eigenvalues, Determined by spectrum, Switching

\end{abstract}

\section{Introduction}

A major problem in the field of spectral graph theory is to decide whether a given graph is determined by its spectrum (the eigenvalues of its adjacency matrix). This problem has been solved for many families of graphs; sometimes by proving that the spectrum determines the graph, and sometimes by constructing cospectral mates (non-isomorphic graphs with the same spectrum).
Although a conjecture by the third author states that almost all graphs are determined by their spectrum, many well-known graphs seem to have cospectral mates. This is, in particular, the case for several graphs in the Johnson and Grassmann association schemes,  see for instance \cite{cioba,spence,koolen, ramezani,ihringer}.
We contribute to this line of research by constructing cospectral mates for a number of open cases in the Johnson and Grassmann schemes. For this, we use switching techniques due to Godsil and McKay \cite{gm} and Wang, Qiu and Hu \cite{wqh}.

This paper is structured as follows. In \rref{Section}{sec:preliminaries}, we introduce the necessary definitions. \rref{Section}{sec:previousresults} contains an overview of previous results on the cospectrality of generalized Johnson and Grassmann graphs.
In \rref{Section}{sec:j2n4}, we 
prove that the generalized Johnson graph \(J_{\{2\}}(n,4)\), \(n\geq8\), is not determined by its spectrum, thereby solving the second open problem of \cite{cioba}. In \rref{Section}{sec:j2kk}, we 
show that the generalized Johnson graph \(J_{\{1,2,\dots\frac{k-1}{2}\}}(2k,k)\) with \(k\geq5\), \(k\) odd, is not determined by its spectrum. In \rref{Section}{sec:k2nk}, we 
prove that the \(q\)-Kneser graph \(K_q(n,k)\) is not determined by its spectrum if \(q=2\).
In \rref{Section}{sec:computational} we provide three specific generalized Johnson graphs that are not determined by their spectrum and pose some open questions. Finally, the \hyperref[appendix]{Appendix} contains an overview of all cospectrality results for small graphs in the aforementioned schemes.

\section{Preliminaries}\label{sec:preliminaries}

Let \(\Gamma=(V,E)\) denote a simple loopless graph.
The \emph{spectrum} of \(\Gamma\) is the multiset of eigenvalues of the adjacency matrix of \(\Gamma\). Graphs are \emph{cospectral} if they have the same spectrum. Two graphs that are cospectral but not isomorphic, are called \emph{cospectral mates}. A graph is \emph{determined by its spectrum} (DS) if it has no cospectral mate. Otherwise, it is not determined by its spectrum (NDS).


One of the most well-known tools for constructing cospectral graphs was introduced by Godsil and McKay \cite{gm}.

\begin{theorem}[GM-switching \cite{gm}]\label{thm:gmgeneral}
    Let \(\Gamma\) be a graph and \(\{C_1,C_2,\dots C_k,D\}\) a partition of \(V(\Gamma)\) such that the following hold for all \(i,j\in\{1,2,\dots,k\}\):
    \begin{enumerate}[(i)]
        \item Any two vertices of \(C_i\) have the same number of neighbours in \(C_j\).
        \item Every vertex in \(D\) has exactly \(0\), \(\frac{1}{2}|C_i|\) or \(|C_i|\) neighbours in \(C_i\). 
    \end{enumerate}
    For all \(i\in\{1,2,\dots,k\}\), every \(u\in C_i\) and every \(v\in D\) that has exactly \(\frac{1}{2}|C_i|\) neighbours in \(C_i\), reverse the adjacency between \(u\) and \(v\). The resulting graph is cospectral with \(\Gamma\). We say that it is obtained from \(\Gamma\) by \emph{GM-switching} with respect to \(C_1,C_2,\dots C_k\).
\end{theorem}




Another switching technique to construct cospectral graphs was recently introduced by Wang, Qui and Hu \cite{wqh}.

\begin{theorem}[WQH-switching \cite{wqh}]\label{thm:wqh}
    Let \(\Gamma\) be a graph and \(C_1,C_2\) disjoint subsets of \(V(\Gamma)\) such that the following hold:
    \begin{enumerate}[(i)]
        \item \(|C_1|=|C_2|.\)
        \item There exists a constant \(c\) such that for all \(i,j\in\{1,2\}\), \(i\neq j\) and every vertex of \(C_i\), the number of neighbours in \(C_i\) minus the number of neighbours in \(C_j\) is equal to \(c\).
        \item Every vertex outside \(C_1\cup C_2\) has either:
        \begin{enumerate}
            \item \(|C_1|\) neighbours in \(C_1\) and \(0\) neighbours in \(C_2\),
            \item \(0\) neighbours in \(C_1\) and \(|C_2|\) neighbours in \(C_2\),
            \item the same number of neighbours in \(C_1\) and \(C_2\).
        \end{enumerate}
    \end{enumerate}
    For every \(u\in C_1\cup C_2\) and every \(v\notin C_1\cup C_2\) for which (a) or (b) holds, reverse the adjacency between \(u\) and \(v\). The resulting graph is cospectral with \(\Gamma\). We say that it is obtained from \(\Gamma\) by \emph{WQH-switching} with respect to \(C_1\) and \(C_2\).
\end{theorem}

We should mention that a more general version of WQH-switching was introduced in \cite{QJW2020}.

Let \(n\), \(k\) be positive integers with \(k\leq n\). We now define the graphs of our interest, namely those in the Johnson and Grassmann schemes.

\begin{definition}\label{def:johnson}
    Let \(S\subseteq\{0,1,\dots,k-1\}\). The \emph{generalized Johnson graph \(J_S(n,k)\)} has as vertices the \(k\)-subsets of \(\{1,\dots,n\}\), where two vertices are adjacent if their intersection size is in \(S\).
\end{definition}

In particular,
\(J_{\{0\}}(n,k)\) is the \emph{Kneser graph} \(K(n,k)\) and \(J_{\{k-1\}}(n,k)\) is the \emph{Johnson graph} \(J(n,k)\).


If we replace sets by subspaces and sizes by dimensions in \rref{Definition}{def:johnson}, we obtain the definition of a generalized Grassmann graph. A vector space of dimension \(k\) is called a \emph{\(k\)-space} for short. Furthermore, let \(\mathbb{F}_q\) denote the finite field of order \(q\) (\(q\) is a prime power).

\begin{definition}
    Let \(S\subseteq\{0,1,\dots,k-1\}\). The \emph{generalized Grassmann graph \(J_{q,S}(n,k)\)} has as vertices the \(k\)-subspaces of \(\mathbb{F}_q^n\), where two vertices are adjacent if their intersection dimension is in \(S\).
\end{definition}

In particular, \(J_{q,\{0\}}(n,k)\) is the \emph{\(q\)-Kneser graph} \(K_q(n,k)\) and \(J_{q,\{k-1\}}(n,k)\) is the \emph{Grassmann graph} \(J_q(n,k)\).

    From now on, for all the aforementioned graphs, we assume that \(k\leq n/2\), since \(J_S(n,k)\) is isomorphic to \(J_{\{s+n-2k\mid s\in S\}}(n,n-k)\). 
    We also assume that \(|S|\leq k/2\) because the complement of \(J_S(n,k)\) is \(J_{\{1,\dots,n\}\setminus S}(n,k)\), and a regular graph is DS if and only if its complement is DS. 
    Note that these observations also hold for Grassmann graphs.
    Also note that if \(S=\emptyset\), the graph is complete or edgeless and therefore trivially DS. Therefore, we also assume that \(|S|\geq1\), and in particular, \(k\geq2\).

\section{Previous results}\label{sec:previousresults}
The following results about the cospectrality of generalized Johnson and Grassmann graphs have been obtained by various authors. We present them in a chronological order.

Let \(q\) be a prime power and \(1\leq k\leq n/2\).



\begin{theorem}[\cite{chang,hoffman}]\label{chang}
    \(K(n,2)\) is DS if and only if \(n\neq8\).
\end{theorem}


\begin{theorem}[\cite{liu}]\label{liu}
    \(K(2k+1,k)\) is DS.
\end{theorem}

\begin{theorem}[\cite{koolen}]\label{koolen}
    \(J_q(2k+1,k)\) is NDS.
\end{theorem}

\begin{theorem}[\cite{spence}]\label{spence}
    \(J(n,k)\) and \(J_q(n,k)\) are NDS if \(k\geq 3\).
\end{theorem}

\begin{theorem}[\cite{ramezani}]\label{ramezani}$ $
    \begin{enumerate}[(i)]
        \item \(K(n,k)\) is NDS if there exists an \(m<k-1\) for which \(\binom{n-k+m}{m}=2\binom{n-2k+m}{m}\).
        \item \(K_{\{0,2,4,\dots\}}(n,k)\) is NDS if \(k\geq3\).
    \end{enumerate}
\end{theorem}

\begin{theorem}[\cite{cioba}]\label{cioba}$ $
    \begin{enumerate}[(i)]
    \item \(J_{\{0,1,\dots,m\}}(3k-2m-1,k)\) is NDS if \(k\geq\max(m+2,3)\).
    \item \(J_{\{0,1,\dots,m\}}(n,2m+1)\) is NDS if \(m\geq2\) and \(n\geq4m+2\).
    \end{enumerate}
\end{theorem}

\begin{theorem}[\cite{ihringer}]\label{ihringer}
    \(J_q(n,2)\) is NDS.
\end{theorem}


\section{\texorpdfstring{\(\boldsymbol{J_{\{2\}}(n,4)}\) is NDS if \(\boldsymbol{n\geq8}\)}{J\{2\}(n,4) is NDS if n≥8}}\label{sec:j2n4}

The authors of \cite{cioba} found, computationally, that the graph \(J_{\{2\}}(8,4)\) has a cospectral mate by GM-switching with respect to switching sets of size \(8\), which implies that it is NDS. 
In this section, we place this graph in an infinite family of graphs that are NDS by using the recent WQH-switching. This solves the second open problem of \cite{cioba}.

    We assume that \(n\geq8\), in accordance with the usual assumption that \(k\leq n/2\). Note that if \(n=7\), \(J_{\{2\}}(n,4)\cong J_{\{1\}}(7,3)\) is NDS by \autoref{koolen}. If \(n=6\), then \(J_{\{2\}}(n,4)\cong K(6,2)\) is DS by \autoref{chang} and if \(n\leq5\), then \(J_{\{2\}}(n,4)\) is trivially DS.

In the main result of this section, we use that \(J_{\{2\}}(n,4)\) is edge-regular (adjacent vertices have a constant number of common neighbours), while the  cospectral mate is not. Actually, every ``elementary'' generalized Johnson graph \(J_{\{i\}}(n,k)\) is edge-regular, because it is the distance-\((k-i)\) graph of the Johnson graph \(J(n,k)\), which is distance-regular. 

\begin{lemma}\label{lemma:edgereg}
    Adjacent vertices of \(J_{\{2\}}(n,4)\) have \(\frac{1}{2}n(n+3)-26\) common neighbours.
\end{lemma}
\begin{proof}
    Choose any two adjacent vertices \(v\) and \(w\). By relabelling the elements of \linebreak \(\{1,2,\dots,n\}\), we may assume that \(v=\{1,2,3,4\}\) and \(w=\{1,2,5,6\}\). There are three types of common neighbours.
    \begin{enumerate}[(i)]
        \item Vertices that include \(\{1,2\}\). The other two elements can be chosen from \(\{7,8,\dots,n\}\). So there are \(\binom{n-6}{2}\) such vertices.
        \item Vertices that intersect \(\{1,2\}\) in one element. They must contain one element of \(\{3,4\}\) and one element of \(\{5,6\}\), while the remaining element should be in \(\{7,8,\dots,n\}\). There are \(2^3(n-6)\) such vertices.
        \item If a vertex is adjacent to \(v\) and \(w\) but does not contain \(1\) or \(2\), it must be equal to \(\{3,4,5,6\}\).
    \end{enumerate}
    We can now conclude that \(v\) and \(w\) have \(\binom{n-6}{2}+2^3(n-6)+1=\frac{1}{2}n(n+3)-26\) common neighbours.
\end{proof}

\begin{theorem}\label{thm:j2n4}
    \(J_{\{2\}}(n,4)\) is not determined by its spectrum if \(n\geq8\).
\end{theorem}

\begin{proof}
Define the sets \(C_1:=\{v_1,v_2,v_3\}\) and \(C_2:=\{v_4,v_5,v_6\}\), where:\\[5mm]
    \begin{minipage}{.5\textwidth}
    \centering
    \begin{itemize}
        \item \(v_1 := \{1,2,3,4\}\)
        \item \(v_2 := \{1,2,3,5\}\)
        \item \(v_3 := \{1,2,3,6\}\)
        \item \(v_4 := \{1,4,5,6\}\)
        \item \(v_5 := \{2,4,5,6\}\)
        \item \(v_6 := \{3,4,5,6\}\)
    \end{itemize}
    \end{minipage}
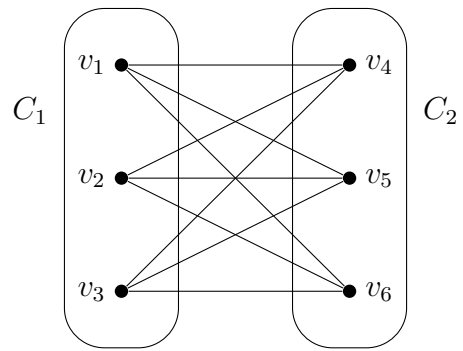
    \begin{minipage}{.5\textwidth}
        \centering
        \newcommand{\radius}{1.5}
\begin{tikzpicture}[remember picture]
    \draw[rounded corners=15pt] (-1.5*\radius,-1.5*\radius) rectangle (-.5*\radius,1.5*\radius)
    (.5*\radius,-1.5*\radius) rectangle (1.5*\radius,1.5*\radius);
	\path[every node/.append style={circle, fill=black, minimum size=5pt, label distance=2pt, inner sep=0pt}]
    (-\radius,\radius) node[label={180:\(v_1\)}] (0) {}
    (-\radius,0) node[label={180:\(v_2\)}] (1) {}
    (-\radius,-\radius) node[label={180:\(v_3\)}] (2) {}
    (\radius,\radius) node[label={0:\(v_4\)}] (3) {}
    (\radius,0) node[label={0:\(v_5\)}] (4) {}
    (\radius,-\radius) node[label={0:\(v_6\)}] (5) {};
    \draw (0) edge (3) edge (4) edge (5)
    (1) edge (3) edge (4) edge (5)
    (2) edge (3) edge (4) edge (5);
    \draw (-1.8*\radius,.6*\radius) node {\(C_1\)}
    (1.8*\radius,.6*\radius) node {\(C_2\)};
    \end{tikzpicture}
    \end{minipage}\\[5mm]
Notice the symmetry with respect to the sets \(\{1,2,3\}\) and \(\{4,5,6\}\). With this in mind, we can prove that \(C_1\) and \(C_2\) form a WQH-switching set of \(J_{\{2\}}(n,4)\), by checking all conditions of \autoref{thm:wqh}. The first two conditions are fulfilled by definition. In order to check the third condition, select an arbitrary \(4\)-set \(v\notin C_1\cup C_2\). We distinguish seven cases, according to the intersection size of \(v\) with \(\{1,2,3\}\) and \(\{4,5,6\}\), see \autoref{tab:new2}. Because of the symmetry of \(C\) with respect to these sets, we assume that \(|v\cap\{1,2,3\}|\leq|v\cap\{4,5,6\}|\). Also note that \(|v\cap\{1,2,3\}|+|v\cap\{4,5,6\}|\leq4\), and that these cardinalities cannot be \(1\) and \(3\) simultaneously, since otherwise the vertex is one of the vertices in \(C_1\cup C_2\). In each case, the requirements of \autoref{thm:wqh}(iii) are met.
    \begin{table}[H]
        \centering
        \begin{tabular}{|c|c|c|c|c|c|c|c|c|}
    \hline
    \(|v\cap\{1,2,3\}|\) & 0&0&0&0&1&1&2\\
    \hline
    \(|v\cap\{4,5,6\}|\) & 0&1&2&3&1&2&2\\
    \hlineB{5}
    \# neighbours in \(C_1\) &0&0&0&0&1&2&1\\
    \hline
    \# neighbours in \(C_2\) &0&0&3&0&1&2&1\\
    \hline
\end{tabular}
        \caption{Given \(v\notin C_1\cup C_2\), the number of neighbours of \(v\) in \(C_1\) and \(C_2\) only depend on its intersection size with \(\{1,2,3\}\) and \(\{4,5,6\}\).}
        \label{tab:new2}
    \end{table}
    
    We are left to prove that the graph obtained by WQH-switching with respect to \(C_1\) and \(C_2\), is not isomorphic to the original one.
    
    Consider the vertices \(v=\{1,2,3,4\}\in C_1\) and \(w=\{1,4,5,7\}\notin C_1\cup C_2\). The neighbours of \(w\) are preserved by switching, since \(|w\cap\{1,2,3\}|=1\) and \(|w\cap\{4,5,6\}|=2\) (see \autoref{tab:new2}). In particular, \(v\) and \(w\) are adjacent in both \(\Gamma\) and \(\Gamma'\). We show that they have more common neighbours in \(\Gamma'\) than in \(\Gamma\). For this, we only need to consider neighbours of \(w\) for which the adjacency with \(v\) is changed during the switching process. We observe from \autoref{tab:new2} that the switching only affects those vertices which have \(2\) elements in one of the sets \(\{1,2,3\}, \{4,5,6\}\) and none in the other. Neighbours of \(w\) that contain \(2\) elements of \(\{1,2,3\}\) and none of \(\{4,5,6\}\), are of the form \(\{1,2,7,a\}\) or \(\{1,3,7,a\}\) with \(a\geq8\). So there are a total of \(2(n-7)\) vertices that are adjacent to \(v\) and \(w\) in \(\Gamma\) but not adjacent to \(v\) in \(\Gamma'\). On the other hand, neighbours of \(w\) that contain no element of \(\{1,2,3\}\) and \(2\) elements of \(\{4,5,6\}\), are of the form \(\{4,5,a,b\}\), \(\{4,6,7,a\}\) or \(\{5,6,7,a\}\) with \(a,b\geq8\). They are not adjacent to \(v\) in \(\Gamma\), but become adjacent to \(v\) after switching. So \(\binom{n-7}{2}+2(n-7)\) new common neighbours are created.
    
    We conclude that \(v\) and \(w\) have \(\binom{n-7}{2}\) more common neighbours in \(\Gamma'\) than any two adjacent vertices in \(\Gamma\), which is a strictly positive difference, except when \(n=8\). In order to solve this one specific case, we take a look at an other pair of vertices.
    
    Assume \(n=8\) and consider the vertices \(v=\{1,2,3,4\}\in C_1\) and \(u=\{5,6,7,8\}\notin C_1\cup C_2\). They are not adjacent in \(\Gamma\), but become adjacent after switching, since \(|u\cap\{1,2,3\}|=0\) and \(|u\cap\{4,5,6\}|=2\) (see \autoref{tab:new2}). We give a lower bound on the number of common neighbours of \(v\) and \(u\) in \(\Gamma'\). In the original graph, \(v\) and \(u\) have \(\binom{4}{2}^2=36\) common neighbours. Three of these are the vertices of \(C_2\), and are therefore lost after switching. All other common neighbours of \(v\) and \(u\) in \(\Gamma\) that are no longer common neighbours in \(\Gamma'\), are those that contain \(2\) elements of \(\{1,2,3\}\) and none of \(\{4,5,6\}\), i.e.\ the three vertices \(\{1,2,7,8\}\), \(\{1,3,7,8\}\) and \(\{2,3,7,8\}\). We are left with at least \(36-3-3=30\) common neighbours, which is strictly more than \(\frac{1}{2}8(8+3)-26=18\), the number of common neighbours of any two vertices in \(\Gamma\) (\rref{Lemma}{lemma:edgereg}).
    
    We proved that there are always two neighbours in \(\Gamma'\) which have strictly more common neighbours than any two neighbours in \(\Gamma\). So \(\Gamma\) and \(\Gamma'\) cannot be isomorphic.
\end{proof}

\section{\texorpdfstring{\(\boldsymbol{J_{\{1,2,\dots\frac{k-1}{2}\}}(2k,k)}\) is NDS if \(\boldsymbol{k\geq5}\), \(\boldsymbol{k}\) odd}{J\{1,2,\dots,(k-1)/2\}(2k,k) is NDS if k≥5, k odd}}\label{sec:j2kk}

The next result uses GM-switching with respect to two switching sets. Our argument is inspired by the proof of \autoref{cioba}(i), as \(J_{\{1,2,\dots\frac{k-1}{2}\}}(2k,k)\) can be seen as two copies of \(J_{\{0,1,\dots\frac{k-3}{2}\}}(2k-1,k-1)\), connected in a certain way. While our first switching set corresponds to the one used to show \autoref{cioba}(i), the second one is given by its complement.

\begin{theorem}\label{thm:j2kk}
    \(J_{\{1,2,\dots\frac{k-1}{2}\}}(2k,k)\) is NDS if \(k\geq5\), \(k\) odd.
\end{theorem}
\begin{proof}
    Denote \(\Gamma:=J_{\{1,2,\dots\frac{k-1}{2}\}}(2k,k)\). Let \(C_1\) be the set of \(k\)-subsets of \([2k]\) that contain \([k-1]\) and let \(C_2\) be the set of \(k\)-subsets of \([2k]\) that are disjoint with \([k-1]\). There is a bijection from \(C_1\) to \(C_2\), given by the complement.

    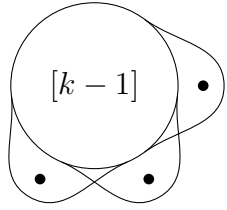
\begin{figure}[H]
        \centering
        \newcommand{\radius}{1.1}
\begin{tikzpicture}
    \draw[color=black] (0,0) node {\([k-1]\)} circle (\radius) {};
    \draw[rotate around={-0:(0,0)}] (.6*\radius,-.8*\radius) to[out=36,in=-90] (1.55*\radius,0) (.6*\radius,.8*\radius) to[out=-36,in=90] (1.55*\radius,0);
    \draw[rotate around={-120:(0,0)}] (.6*\radius,-.8*\radius) to[out=36,in=-90] (1.55*\radius,0) (.6*\radius,.8*\radius) to[out=-36,in=90] (1.55*\radius,0);
    \draw[rotate around={-60:(0,0)}] (.6*\radius,-.8*\radius) to[out=36,in=-90] (1.55*\radius,0) (.6*\radius,.8*\radius) to[out=-36,in=90] (1.55*\radius,0);
    \path[every node/.append style={circle, fill=black, minimum size=4pt,inner sep=0pt}]
    (-0:1.3*\radius) node {} (-120:1.3*\radius) node {} (-60:1.3*\radius) node {};
\end{tikzpicture}
        \caption{Elements of \(C_2\) include \([k-1]\) and one other element.}
    \end{figure}
    
    We first check the conditions of \autoref{thm:gmgeneral} in order to show that \(C_1\) and \(C_2\) form a GM-switching set of \(\Gamma\). It is immediate that \(|C_1|=|C_2|=k+1\). Moreover, every two \(k\)-sets of \(C_1\) (resp.\ \(C_2\)) intersect in \(k-1\) elements, which means that \(C_1\) (resp.\ \(C_2\)) is a coclique. Any vertex in \(C_1\) has \(k\) neighbours in \(C_2\) and vice versa. Choose an arbitrary vertex \(v\notin C_1\cup C_2\) and let \(l\) be the size of its intersection with \([k-1]\). Then \(v\) intersects \(l+1\) vertices of \(C_1\) in \(l\) elements and the other \(k-l\) vertices in \(l+1\) elements. Similarly, \(v\) intersects \(l+1\) vertices of \(C_2\) in \(k-l\) elements and \(k-l\) vertices in \(k-l-1\) elements. So if \(l<\frac{k-1}{2}\), then \(v\) is adjacent to all vertices of \(C_1\) and none of \(C_2\). If \(l=\frac{k-1}{2}\), then it is adjacent to \(\frac{k+1}{2}=\frac{1}{2}|C_1|\) elements of each switching set, and if \(l>\frac{k-1}{2}\), then it is adjacent with no vertices of \(C_1\) and all vertices of \(C_2\). We conclude that \(C_1\) and \(C_2\) fulfil the GM-switching conditions.
    
    Next, we show that the graph \(\Gamma'\) obtained from \(\Gamma\) by GM-switching with respect to \(C_1\) and \(C_2)\), is not isomorphic to \(\Gamma\). Define \(v:=[k]\in C_1\) and \(w:=\{2,3,\dots,k+1\}\notin C_1\cup C_2\).
    
    We claim that \(v\) and \(w\) have strictly less common neighbours after switching. For this, it suffices to count the number of common neighbours that are lost or added during the GM-switching process. Therefore, we only consider those neighbours of \(w\) that have \(\frac{1}{2}|C_1|=\frac{k+1}{2}\) neighbours in \(C_1\), i.e.\ those that contain exactly \(\frac{k-1}{2}\) elements of \([k-1]\).
    First consider the common neighbours that are added. They should not be adjacent to \(v\) in the original graph, which means that they must contain the element \(k\). Adjacency with \(w\) implies that they must also contain the element \(1\) and not the element \(k+1\). Thus, there are
    \[\underbrace{\binom{k-2}{\frac{k-3}{2}}}_{\substack{\text{\(\frac{k-1}{2}\) elements of \([k-1]\),}\\ \text{one of which is \(1\)}}}
    \cdot
    \underbrace{\binom{k-1}{\frac{k-1}{2}}}_{\substack{\text{\(\frac{k+1}{2}\) elements not in \([k-1]\)}\\ \text{one of which is \(k\)}}}\] added common neighbours of \(v\) and \(w\).
    Now consider the common neighbours that are lost. Since these vertices must originally be adjacent to \(v\), they may not contain \(k\). If they contain the element \(1\), then they will automatically be adjacent to \(w\). So there are at least
    \[\underbrace{\binom{k-2}{\frac{k-3}{2}}}_{\substack{\text{\(\frac{k-1}{2}\) elements of \([k-1]\),}\\ \text{one of which is \(1\)}}}
    \cdot
    \underbrace{\binom{k}{\frac{k+1}{2}}}_{\substack{\text{\(\frac{k+1}{2}\) elements}\\ \text{not in \([k]\)}}}
    \]
    common neighbours of \(v\) and \(w\) that are lost during switching, which is already strictly more than the number of added common neighbours.
    
    Now, assume by contradiction that \(\Gamma\cong\Gamma'\). Let \(\lambda(u,w)\) denote the number of common neighbours of \(u\) and \(w\), for any vertex \(u\) and \(w\) defined as before. Since \(\Gamma\) is vertex-transitive, the multiset
    \[\Lambda(w)=\{\lambda(u,w)\mid u\not\sim w\}\]
    remains the same after switching. The value of \(\lambda(u,w)\) stays the same if \(u\notin C_1\cup C_2\), so we can restrict ourselves to vertices in \(C_1\cup C_2\), i.e.\ the multiset
    \[\{\lambda(u,w)\mid u\in C_1\cup C_2, u\not\sim w\}\]
    remains the same. Since \(w\) is adjacent to all vertices of \(C_2\) and none of \(C_1\), this is equivalent to saying that \[\{\lambda(u,w)\mid u\in C_1\}\] is invariant. We now apply the same reasoning as in \cite{cioba}. Define \(v':=[k-1]\cup\{k+1\}\). The automorphism induced by the permutation \((k+2\ k+3\ \cdots 2k)\) fixes \(v\), \(v'\) and \(w\), but shifts the other vertices of \(C_1\) cyclically. So \(\lambda(u,w)\) is constant among all \(u\in C_1\setminus\{v,v'\}\). Together with the observation that \(\lambda(v,w)=\lambda(v',w)\), we see that \(\lambda(v,w)\) is the same for \(\Gamma\) and \(\Gamma'\), a contradiction with the above.
    
    We conclude that \(\Gamma\) and \(\Gamma'\) are cospectral mates. Thus, \(\Gamma\) is NDS.
\end{proof}

\section{\texorpdfstring{\(\boldsymbol{q}\)-Kneser graphs are NDS if \(\boldsymbol{q=2}\)}{q-Kneser graphs are NDS if q=2}}\label{sec:k2nk}

In this section, we prove that the \(q\)-Kneser graph \(K_2(n,k)=J_{2,\{0\}}(n,k)\) is NDS (under the assumption that \(k\geq2\), otherwise the graph is trivially DS) using GM-switching with respect to a switching set of size \(4\). This result can be seen as a \(q\)-analog of \autoref{ramezani}(i) in the case that \(q=2\) (and without the extra condition).

\begin{theorem}\label{thm:k2nk}
    \(K_2(n,k)\) is not determined by its spectrum.
\end{theorem}
\begin{proof}
    Let \(\Gamma:=K_2(n,k)\). Define the following \(1\)-spaces in \(\mathbb{F}_2^n\) (drawn projectively, as points in \(PG(n-1,2)\)).\\[3mm]
    \begin{minipage}{.39\textwidth}
    \begin{itemize}
        \item \(p_1 := \langle(1,0,0,0,0,\dots,0)\rangle\)\vspace{-1mm}
        \item \(p_2 := \langle(0,1,0,0,0,\dots,0)\rangle\)\vspace{-1mm}
        \item \(p_3 := \langle(0,0,1,0,0,\dots,0)\rangle\)\vspace{-1mm}
        \item \(p_4 := \langle(1,1,0,0,0,\dots,0)\rangle\)\vspace{-1mm}
        \item \(p_5 := \langle(1,0,1,0,0,\dots,0)\rangle\)\vspace{-1mm}
        \item \(p_6 := \langle(0,1,1,0,0,\dots,0)\rangle\)\vspace{-1mm}
    \end{itemize}
    \end{minipage}
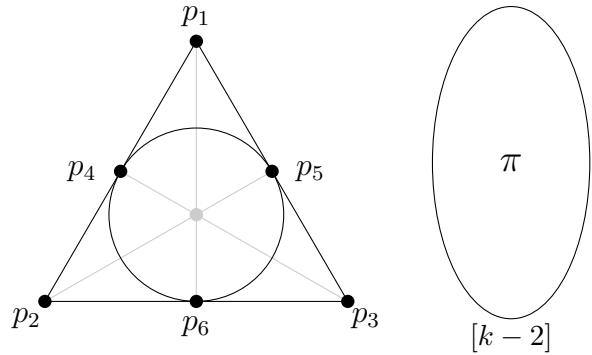
    \begin{minipage}{.6\textwidth}
        \centering
        \newcommand{\radius}{2.3}
\vspace{-5mm}
\begin{tikzpicture}[remember picture]
\draw (-30:\radius) -- coordinate (P1)
(90:\radius) -- coordinate (P2)
(210:\radius) -- coordinate (P3) cycle;
\draw[black!20] (210:\radius) -- (P1) (-30:\radius) -- (P2) (90:\radius) -- (P3);
\node[draw] at (0,0) [circle through=(P1)] {};
\path[every node/.append style={circle, fill=black, minimum size=5pt, label distance=0pt, inner sep=0pt}]
(0,0) node[color=black!20] {}
(P1) node[label={[label distance=5pt]0:\(p_5\)}] {}
(P2) node[label={[label distance=5pt]180:\(p_4\)}] {}
(P3) node[label={[label distance=0pt]270:\(p_6\)}] {}
(-30:\radius) node[label={[label distance=0pt]-40:\(p_3\)}] {}
(90:\radius) node[label={[label distance=0pt]90:\(p_1\)}] {}
(210:\radius) node[label={[label distance=0pt]220:\(p_2\)}] {};
\draw (1.8*\radius,.3*\radius) node[] {\large\(\pi\)} ellipse ({.45*\radius} and .9*\radius) {};
\draw (1.8*\radius,-.71*\radius) node {\small\([k-2]\)};
\end{tikzpicture}
    \end{minipage}\\[5mm]
    Let \(\pi\) be a \((k-2)\)-space that has a trivial intersection with the \(3\)-space (projective plane) spanned by these points. Define \(C:=\{p_1p_2\pi,\;p_1p_3\pi,\;p_2p_3\pi,\;p_4p_5\pi\}\).
    
    We first prove that \(C\) is a GM-switching set of \(\Gamma\) by checking the conditions of \autoref{thm:gmgeneral} (where \(k=1\)). Every element of \(C\) meets every other element of \(C\) in a \(k-1\)-space, so \(C\) is a coclique and therefore regular. Choose an arbitrary \(k\)-space not in \(C\) and let \(R:=\{p_1\pi,p_2\pi,\dots,p_6\pi\}\). 
    If the chosen \(k\)-space intersects all elements of \(R\) trivially, then it also intersects all spaces in \(C\) trivially, which means that it is adjacent to all the corresponding vertices of the \(q\)-Kneser graph. If it intersects just one of the elements of \(R\) non-trivially, then it is adjacent to exactly \(2\) vertices of \(C\). If the \(k\)-space intersects at least two elements of \(R\) non-trivially, then it either intersects \(\pi\) nontrivially, or it contains a \(2\)-space in the \(k+1\)-space \(p_1p_2p_3\pi\). Either case, it intersects every element of \(C\) non-trivially, therefore being adjacent to no element of \(C\).
    
    Let \(\Gamma'\) be the graph obtained from \(\Gamma\) by GM-switching with respect to \(C\). We prove that \(\Gamma\) and \(\Gamma'\) are non-isomorphic.
    Let \(\tau\) be a \((k-1)\)-space that intersects \(p_1p_2p_3\pi\) trivially and consider the pairwise non-adjacent \(k\)-spaces \(p_1\tau\), \(p_2\tau\) and \(p_4\tau\) in \(\Gamma'\). Then there is exactly one \(k\)-space that is adjacent to \(p_1\tau\) but not adjacent to both \(p_2\tau\) and \(p_4\tau\) in \(\Gamma'\). This is the \(k\)-space \(p_1p_3\pi\). Indeed, a \(k\)-space outside \(C\) that intersects both \(p_2\tau\) and \(p_4\tau\), must intersect \(p_1\tau\) as well. So the only spaces that can meet this property are in \(C\). Since the switching reverses adjacency for \(p_1\tau\), \(p_2\tau\) and \(p_4\tau\), we get that \(p_1p_3\pi\) is adjacent to \(p_1\tau\) and not adjacent to \(p_2\tau\) and \(p_4\tau\), while the other elements of \(C\) are not.
    
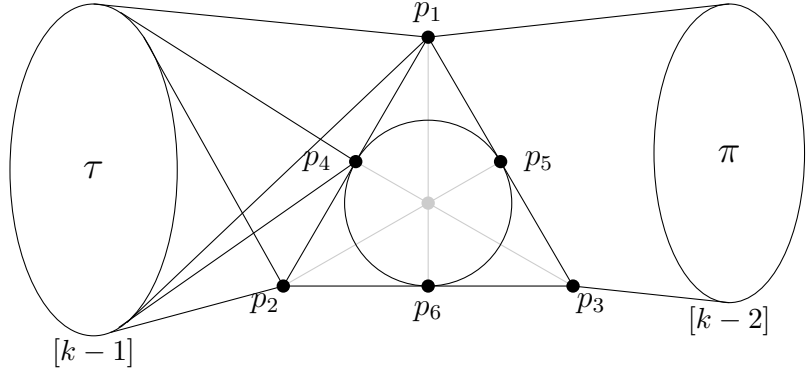
\vspace{4mm}
    \begin{figure}[H]
        \centering
        \newcommand{\radius}{2.2}
\vspace{-5mm}
\begin{tikzpicture}[remember picture]
\draw (-30:\radius) -- coordinate (P1)
(90:\radius);
\draw (90:\radius) -- coordinate (P2)
(210:\radius) -- coordinate (P3) (-30:\radius);
\draw[black!20] (210:\radius) -- (P1) (-30:\radius) -- (P2) (90:\radius) -- (P3);
\node[draw] at (0,0) [circle through=(P1)] {};

\draw (90:\radius) edge (-2*\radius,1.2*\radius) edge (-1.8*\radius,-.72*\radius);
\draw (210:\radius) edge (-1.7*\radius,1*\radius) edge (-1.9*\radius,-.78*\radius)
(P2) edge (-1.84*\radius,1.15*\radius) edge (-1.85*\radius,-.75*\radius);
\draw (90:\radius) edge (1.8*\radius,1.2*\radius)
(-30:\radius) edge (1.8*\radius,-.6*\radius);

\path[every node/.append style={circle, fill=black, minimum size=5pt, label distance=0pt, inner sep=0pt}]
(0,0) node[color=black!20] {}
(P1) node[label={[label distance=5pt]0:\(p_5\)}] {}
(P2) node[label={[label distance=5pt]180:\(p_4\)}] {}
(P3) node[label={[label distance=0pt]270:\(p_6\)}] {}
(-30:\radius) node[label={[label distance=0pt]-40:\(p_3\)}] {}
(90:\radius) node[label={[label distance=0pt]90:\(p_1\)}] {}
(210:\radius) node[label={[label distance=0pt]220:\(p_2\)}] {};

\draw (1.8*\radius,.3*\radius) node[] {\large\color{black}\(\pi\)} ellipse ({.45*\radius} and .9*\radius) {};
\draw (-2*\radius,.2*\radius) node {\large\color{black}\(\tau\)} ellipse ({.5*\radius} and 1*\radius) {};
\draw (-2*\radius,-.9*\radius) node {\small\([k-1]\)};
\draw (1.8*\radius,-.71*\radius) node {\small\([k-2]\)};

\end{tikzpicture}
    \end{figure}
    \vspace{-3mm}
    Now consider three arbitrarily chosen non-adjacent \(k\)-spaces in \(\Gamma\). Call them \(\alpha\), \(\beta\) and \(\gamma\). Non-adjacency here means that they intersect one another. We prove that the number of \(k\)-spaces that intersects \(\alpha\) trivially but intersects both \(\beta\) and \(\gamma\) non-trivially, is never equal to one. Let \(\delta\) be such a space (if it does not exist, we are done). Choose two \(1\)-spaces (projective points) \(p\in\beta\cap\delta\) and \(q\in\gamma\cap\delta\). The \(2\)-space (projective line) \(pq\) does not intersect \(\alpha\), as it lies in \(\delta\). There are \(2^{k(k-2)}\gauss{n-k-1}{k-2}{2}>1\) different \(k\)-spaces through \(pq\) that intersect \(\alpha\) trivially, and they all meet the above property.
    
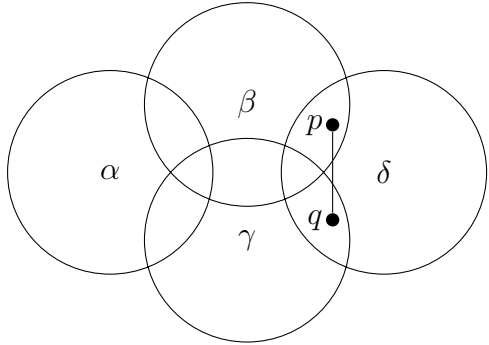
\vspace{-1mm}
    \begin{figure}[H]
        \centering
        \begin{tikzpicture}[scale=.9, remember picture]
    \path[every node/.append style={circle, fill=black, minimum size=5pt, label distance=-1pt, inner sep=0pt}]
    (1.25,.7) node[label={180:\(p\)}] (X) {}
    (1.25,-.7) node[label={180:\(q\)}] (Y) {};
    \draw (-2,0) node {\(\alpha\)} circle (1.5)
    (0,1) node {\(\beta\)} circle (1.5)
    (0,-1) node {\(\gamma\)} circle (1.5)
    (2,0) node {\(\delta\)} circle (1.5)
    (X) edge (Y);
\end{tikzpicture}
        \label{fig:q2switch2}
    \end{figure}
    \vspace{-4mm}
    Since the number of \(k\)-spaces with this property is different for \(\Gamma\) and \(\Gamma'\), while it should be invariant under isomorphism, we conclude that \(\Gamma\) and \(\Gamma'\) are not isomorphic.
\end{proof}

\section{Open problems}\label{sec:computational}






Besides the three infinite families of graphs that are NDS, we also found three ``sporadic'' generalized Johnson graphs with a cospectral mate. For each of these three graphs, we provide a pair of switching sets that produces a cospectral mate, see the table below. The fact that they are non-isomorphic was checked by computer.

\vspace{5mm}
\noindent\scalebox{0.923}{\begin{tabular}{|l|c|l|}
    \hline
    graph & method & switching sets\\
    \hline
    \(J_{\{1\}}(11,4)\) & WQH & \thead[cl]{\(C_1=\left\{\{1,2,3,10\},\{4,5,6,10\},\{7,8,9,10\}\right\}\)\\ \(C_2=\left\{\{1,2,3,11\},\{4,5,6,11\},\{7,8,9,11\}\right\}\)}\\
    \hline
    \(J_{\{2,4\}}(10,5)\) & GM & \thead[cl]{\(C_1=\left\{\{1,2,3,4,5\},\{1,2,3,6,7\},\{1,2,4,6,8\},\{1,2,5,7,8\}\right\}\)\\ \(C_2=\left\{\{6,7,8,9,10\},\{4,5,8,9,10\},\{3,5,7,9,10\},\{3,4,6,9,10\}\right\}\)}\\
    \hline
    \(J_{\{2,4\}}(12,6)\) & GM & \thead[cl]{\(C_1=\left\{\{1,2,3,4,5,6\},\{1,2,3,4,7,8\},\{1,2,3,5,7,9\},\{1,2,3,6,8,9\}\right\}\)\\ \(C_2=\left\{\{7,8,9,10,11,12\},\{5,6,9,10,11,12\},\{4,6,8,10,11,12\},\{4,5,7,10,11,12\}\right\}\)}\\
    \hline
\end{tabular}}
\vspace{5mm}

It would be nice to see if any of these switching sets can be extended to an infinite family of generalized Johnson graphs that are NDS.
We believe that \(J_{\{2,4,\dots\}}(2k,k)\) is NDS if \(k\geq4\), following a similar argument as in \cite{ramezani}, but with two switching sets instead of one.

More generally, it remains an open problem to determine the cospectrality of other graphs in the Johnson and Grassmann schemes (entries ``?'' in the tables in the \hyperref[appendix]{Appendix}). The smallest open case is \(K(9,3)\), see also \cite[Section~5]{cioba}.

\subsection*{Acknowledgements}
Aida Abiad is partially supported by the Dutch Research Council through the grant VI.Vidi.213.085 and by the Research Foundation Flanders through the grant 1285921N.
Jozefien D’haeseleer is supported by the Research Foundation Flanders through the grant 1218522N.


\newpage

\section*{Appendix: overview of results}\label{appendix}

This appendix provides a structured overview of the currently known cospectrality results for some small generalized Johnson and Grassmann graphs. New results are indicated with an asterisk (* or **).\\[-5mm]

\begin{center}
\begin{tabular}{lll}
    \textbf{Legend:} & \colorbox{trivial}{\(^0\)Trivial} & \colorbox{chang}{\(^1\)Hoffman/Chang (1959)}\\[1mm]
    \colorbox{liu}{\(^2\)Huang, Liu (1999)} & \colorbox{spence}{\(^3\)Van Dam et al. (2006)} & \colorbox{ramezani}{\(^4\)Haemers, Ramezani (2010)}\\[1mm]
    \colorbox{cioba}{\(^5\)Cioab{\u{a}} et al. (2018)} & \colorbox{new}{*\autoref{thm:j2n4}} & \colorbox{sporadic}{**Sporadic (\rref{Section}{sec:computational})}
\end{tabular}
\end{center}
\vfill
\begin{table}[H]
    \centering
    \begin{tabular}{|c|cV{5}cV{5}c|}
    \hline
    \multicolumn{2}{|cV{5}}{\multirow{2}{*}{\({J_S(n,2)}\)}} & \multicolumn{1}{cV{5}}{\(S\)} & \multicolumn{1}{c|}{\multirow{2}{*}{\(|V|\)}}\\
    \cline{3-3}
    \multicolumn{2}{|cV{5}}{} & \(\{0\}\) & \\
    \hlineB{5}
    \multirow{6}{*}{\(n\)} & 4 & \cellcolor{trivial}DS\(^0\) & 6\\
    \cline{2-4}
    & 5 & \cellcolor{chang}\hyperref[chang]{DS\(^1\)} & 10\\
    \cline{2-4}
    & 6 & \cellcolor{chang}\hyperref[chang]{DS\(^1\)} & 15\\
    \cline{2-4}
    & 7 & \cellcolor{chang}\hyperref[chang]{DS\(^1\)} & 21\\
    \cline{2-4}
    & 8 & \cellcolor{chang}\hyperref[chang]{NDS\(^1\)} & 28\\
    \cline{2-4}
    & 9 & \cellcolor{chang}\hyperref[chang]{DS\(^1\)} & 36\\
    \hline
    \end{tabular}
    \caption{Cospectrality of small generalized Johnson graphs with \(k=2\).}
    \label{tab:johnson2}
\end{table}
\begin{table}[H]
    \centering
    \begin{tabular}{|c|cV{5}c|c|cV{5}c|}
    \hline
    \multicolumn{2}{|cV{5}}{\multirow{2}{*}{\({J_S(n,3)}\)}} & \multicolumn{3}{cV{5}}{\(S\)} & \multicolumn{1}{c|}{\multirow{2}{*}{\(|V|\)}}\\
    \cline{3-5}
    \multicolumn{2}{|cV{5}}{} & \(\{0\}\) & \(\{1\}\) & \(\{2\}\) & \\
    \hlineB{5}
    \multirow{6}{*}{\(n\)} & 6 & \cellcolor{trivial}DS\(^0\) & \cellcolor{ramezani}\hyperref[ramezani]{NDS\(^4\)} & \cellcolor{spence}\hyperref[spence]{NDS\(^3\)} & 20\\
    \cline{2-6}
    & 7 & \cellcolor{liu}\hyperref[liu]{DS\(^2\)} & \cellcolor{ramezani}\hyperref[ramezani]{NDS\(^4\)} & \cellcolor{spence}\hyperref[spence]{NDS\(^3\)} & 35\\
    \cline{2-6}
    & 8 & \cellcolor{ramezani}\hyperref[ramezani]{NDS\(^4\)} & \cellcolor{ramezani}\hyperref[ramezani]{NDS\(^4\)} & \cellcolor{spence}\hyperref[spence]{NDS\(^3\)} & 56\\
    \cline{2-6}
    & 9 & ? & \cellcolor{ramezani}\hyperref[ramezani]{NDS\(^4\)} & \cellcolor{spence}\hyperref[spence]{NDS\(^3\)} & 84\\
    \cline{2-6}
    & 10 & ? & \cellcolor{ramezani}\hyperref[ramezani]{NDS\(^4\)} & \cellcolor{spence}\hyperref[spence]{NDS\(^3\)} & 120\\
    \cline{2-6}
    & 11 & ? & \cellcolor{ramezani}\hyperref[ramezani]{NDS\(^4\)} & \cellcolor{spence}\hyperref[spence]{NDS\(^3\)} & 165\\
    \hline
    \end{tabular}
    \caption{Cospectrality of small generalized Johnson graphs with \(k=3\).}
    \label{tab:johnson3}
\end{table}
\begin{table}[H]
    \centering
    \begin{tabular}{|c|cV{5}c|c|c|c|c|c|cV{5}c|}
    \hline
    \multicolumn{2}{|cV{5}}{\multirow{2}{*}{\({J_S(n,4)}\)}} & \multicolumn{7}{cV{5}}{\(S\)} & \multicolumn{1}{c|}{\multirow{2}{*}{\(|V|\)}}\\
    \cline{3-9}
    \multicolumn{2}{|cV{5}}{} & \(\{0\}\) & \(\{1\}\) & \(\{2\}\) & \(\{3\}\) & \(\{0,1\}\) & \(\{0,2\}\) & \(\{0,3\}\) & \\
    \hlineB{5}
    \multirow{6}{*}{\(n\)} & 8 & \cellcolor{trivial}DS\(^0\) & ? & \cellcolor{new}\hyperref[thm:j2n4]{NDS*} & \cellcolor{spence}\hyperref[spence]{NDS\(^3\)} & ? & \cellcolor{ramezani}\hyperref[ramezani]{NDS\(^4\)} & ? & 70\\
    \cline{2-10}
    & 9 & \cellcolor{liu}\hyperref[liu]{DS\(^2\)} & ? & \cellcolor{new}\hyperref[thm:j2n4]{NDS*} & \cellcolor{spence}\hyperref[spence]{NDS\(^3\)} & \cellcolor{cioba}\hyperref[cioba]{NDS\(^5\)} & \cellcolor{ramezani}\hyperref[ramezani]{NDS\(^4\)} & ? & 126\\
    \cline{2-10}
    & 10 & ? & ? & \cellcolor{new}\hyperref[thm:j2n4]{NDS*} & \cellcolor{spence}\hyperref[spence]{NDS\(^3\)} & ? & \cellcolor{ramezani}\hyperref[ramezani]{NDS\(^4\)} & ? & 210\\
    \cline{2-10}
    & 11 & \cellcolor{ramezani}\hyperref[ramezani]{NDS\(^4\)} & \cellcolor{sporadic}\hyperref[sec:computational]{NDS**} & \cellcolor{new}\hyperref[thm:j2n4]{NDS*} & \cellcolor{spence}\hyperref[spence]{NDS\(^3\)} & ? & \cellcolor{ramezani}\hyperref[ramezani]{NDS\(^4\)} & ? & 330\\
    \cline{2-10}
    & 12 & ? & ? & \cellcolor{new}\hyperref[thm:j2n4]{NDS*} & \cellcolor{spence}\hyperref[spence]{NDS\(^3\)} & ? & \cellcolor{ramezani}\hyperref[ramezani]{NDS\(^4\)} & ? & 495\\
    \cline{2-10}
    & 13 & ? & ? & \cellcolor{new}\hyperref[thm:j2n4]{NDS*} & \cellcolor{spence}\hyperref[spence]{NDS\(^3\)} & ? & \cellcolor{ramezani}\hyperref[ramezani]{NDS\(^4\)} & ? & 715\\
    \hline
    \end{tabular}
    \caption{Cospectrality of small generalized Johnson graphs with \(k=4\).}
    \label{tab:johnson4}
\end{table}

\begin{landscape}
   \begin{center}
\begin{tabular}{lllll}
    \textbf{Legend:} & \colorbox{trivial}{\(^0\)Trivial} & \colorbox{liu}{\(^2\)Huang, Liu (1999)} & \colorbox{koolen}{\(^2\)Van Dam, Koolen (2005)} & \colorbox{spence}{\(^3\)Van Dam et al. (2006)}\\
    \colorbox{ramezani}{\(^4\)Haemers, Ramezani (2010)} & \colorbox{cioba}{\(^5\)Cioab{\u{a}} et al. (2018)} & \colorbox{ihringer}{\(^6\)Ihringer, Munemasa (2019)} & \colorbox{new}{*\autoref{thm:j2kk}} & \colorbox{sporadic}{**Sporadic (\rref{Section}{sec:computational})}
\end{tabular}
\end{center}
    \vfill
    {\tiny{
\begin{table}[H]
    \centering
    \begin{tabular}{|c|cV{5}c|c|c|c|c|c|c|c|c|c|c|c|c|c|cV{5}c|}
    \hline
    \multicolumn{2}{|cV{5}}{\multirow{2}{*}{\({J_S(n,5)}\)}} & \multicolumn{15}{cV{5}}{\(S\)} & \multicolumn{1}{c|}{\multirow{2}{*}{\(|V|\)}}\\
    \cline{3-17}
    \multicolumn{2}{|cV{5}}{} & \(\{0\}\) & \(\{1\}\) & \(\{2\}\) & \(\{3\}\) & \(\{4\}\) & \(\{0,1\}\) & \(\{0,2\}\) & \(\{0,3\}\) & \(\{0,4\}\) & \(\{1,2\}\) & \(\{1,3\}\) & \(\{1,4\}\) & \(\{2,3\}\) & \(\{2,4\}\) & \(\{3,4\}\) & \\
    \hlineB{5}
    \multirow{6}{*}{\(n\)} & 10 & \cellcolor{trivial}DS\(^0\) & ? & ? & ? & \cellcolor{spence}\hyperref[ramezani]{NDS\(^3\)} & ? & ? & ? & ? & \cellcolor{new}\hyperref[thm:j2kk]{NDS*} & \cellcolor{ramezani}\hyperref[ramezani]{NDS\(^4\)} & ? & ? & \cellcolor{sporadic}\hyperref[sec:computational]{NDS**} & \cellcolor{cioba}\hyperref[cioba]{NDS\(^5\)} & 252\\
    \cline{2-18}
    & 11 & \cellcolor{liu}\hyperref[liu]{DS\(^1\)} & ? & ? & ? & \cellcolor{spence}\hyperref[ramezani]{NDS\(^3\)} & ? & ? & ? & ? & ? & \cellcolor{ramezani}\hyperref[ramezani]{NDS\(^4\)} & ? & ? & ? & \cellcolor{cioba}\hyperref[cioba]{NDS\(^5\)} & 462\\
    \cline{2-18}
    & 12 & ? & ? & ? & ? & \cellcolor{spence}\hyperref[ramezani]{NDS\(^3\)} & \cellcolor{cioba}\hyperref[cioba]{NDS\(^5\)} & ? & ? & ? & ? & \cellcolor{ramezani}\hyperref[ramezani]{NDS\(^4\)} & ? & ? & ? & \cellcolor{cioba}\hyperref[cioba]{NDS\(^5\)} & 792\\
    \cline{2-18}
    & 13 & ? & ? & ? & ? & \cellcolor{spence}\hyperref[ramezani]{NDS\(^3\)} & ? & ? & ? & ? & ? & \cellcolor{ramezani}\hyperref[ramezani]{NDS\(^4\)} & ? & ? & ? & \cellcolor{cioba}\hyperref[cioba]{NDS\(^5\)} & 1287\\
    \cline{2-18}
    & 14 & \cellcolor{ramezani}\hyperref[ramezani]{NDS\(^4\)} & ? & ? & ? & \cellcolor{spence}\hyperref[ramezani]{NDS\(^3\)} & ? & ? & ? & ? & ? & \cellcolor{ramezani}\hyperref[ramezani]{NDS\(^4\)} & ? & ? & ? & \cellcolor{cioba}\hyperref[cioba]{NDS\(^5\)} & 2002\\
    \cline{2-18}
    & 15 & ? & ? & ? & ? & \cellcolor{spence}\hyperref[ramezani]{NDS\(^3\)} & ? & ? & ? & ? & ? & \cellcolor{ramezani}\hyperref[ramezani]{NDS\(^4\)} & ? & ? & ? & \cellcolor{cioba}\hyperref[cioba]{NDS\(^5\)} & 3003\\
    \hline
    \end{tabular}
    \caption{Cospectrality of small generalized Johnson graphs with \(k=5\).}
    \label{tab:johnson5}
\end{table}
}}
    \vfill
    \begin{table}[H]
    \centering
    \begin{tabular}{|c|cV{5}cV{5}cV{5}cV{5}cV{5}cV{5}c|}
    \hline
    \multicolumn{2}{|cV{5}}{\multirow{3}{*}{\({J_{q,S}(n,2)}\)}} & \multicolumn{2}{cV{5}}{\(q=2\)} & \multicolumn{2}{cV{5}}{\(q=3\)} & \multicolumn{2}{c|}{\(q=4\)}\\
    \cline{3-8}
    \multicolumn{2}{|cV{5}}{} & \multicolumn{1}{cV{5}}{\(S\)} & \multicolumn{1}{cV{5}}{\multirow{2}{*}{\(|V|\)}} & \multicolumn{1}{cV{5}}{\(S\)} & \multicolumn{1}{cV{5}}{\multirow{2}{*}{\(|V|\)}} & \multicolumn{1}{cV{5}}{\(S\)} & \multicolumn{1}{c|}{\multirow{2}{*}{\(|V|\)}}\\
    \cline{3-3}\cline{5-5}\cline{7-7}
    \multicolumn{2}{|cV{5}}{} & \(\{0\}\) & & \(\{0\}\) & & \(\{0\}\) & \\
    \hlineB{5}
    \multirow{6}{*}{\(n\)} & 4 & \cellcolor{ihringer}\hyperref[ihringer]{NDS\(^6\)} & 35 & \cellcolor{ihringer}\hyperref[ihringer]{NDS\(^6\)} & 130 & \cellcolor{ihringer}\hyperref[ihringer]{NDS\(^6\)} & 357\\
    \cline{2-8}
    & 5 & \cellcolor{ihringer}\hyperref[ihringer]{NDS\(^6\)} & 155 & \cellcolor{koolen}\hyperref[koolen]{NDS\(^2\)} & 1210 & \cellcolor{koolen}\hyperref[koolen]{NDS\(^2\)} & 5797\\
    \cline{2-8}
    & 6 & \cellcolor{ihringer}\hyperref[ihringer]{NDS\(^6\)} & 651 & \cellcolor{ihringer}\hyperref[ihringer]{NDS\(^6\)} & 11 011 & \cellcolor{ihringer}\hyperref[ihringer]{NDS\(^6\)} & 93 093\\
    \cline{2-8}
    & 7 & \cellcolor{ihringer}\hyperref[ihringer]{NDS\(^6\)} & 2667 & \cellcolor{ihringer}\hyperref[ihringer]{NDS\(^6\)} & 99 463 & \cellcolor{ihringer}\hyperref[ihringer]{NDS\(^6\)} & \(\approx10^6\)\\
    \cline{2-8}
    & 8 & \cellcolor{ihringer}\hyperref[ihringer]{NDS\(^6\)} & 10 795 & \cellcolor{ihringer}\hyperref[ihringer]{NDS\(^6\)} & \(\approx10^6\) & \cellcolor{ihringer}\hyperref[ihringer]{NDS\(^6\)} & \(\approx10^7\)\\
    \cline{2-8}
    & 9 & \cellcolor{ihringer}\hyperref[ihringer]{NDS\(^6\)} & 43 435 & \cellcolor{ihringer}\hyperref[ihringer]{NDS\(^6\)} & \(\approx10^7\) & \cellcolor{ihringer}\hyperref[ihringer]{NDS\(^6\)} & \(\approx10^8\)\\
    \hline
    \end{tabular}
    \caption{Cospectrality of small generalized Grassmann graphs with \(k=2\) and \(q=2,3,4\).}
    \label{tab:grassmann2}
\end{table}
    \vfill
    \newpage
    \begin{center}
\begin{tabular}{lll}
    \textbf{Legend:} & \colorbox{spence}{\(^1\)Van Dam et al. (2006)} & \colorbox{new}{*\autoref{thm:k2nk}}
\end{tabular}
\end{center}
    \vfill
    \begin{table}[H]
    \centering
    \begin{tabular}{|c|cV{5}c|c|cV{5}cV{5}c|c|cV{5}cV{5}c|c|cV{5}c|}
    \hline
    \multicolumn{2}{|cV{5}}{\multirow{3}{*}{\({J_{q,S}(n,3)}\)}} & \multicolumn{4}{cV{5}}{\(q=2\)} & \multicolumn{4}{cV{5}}{\(q=3\)} & \multicolumn{4}{c|}{\(q=4\)}\\
    \cline{3-14}
    \multicolumn{2}{|cV{5}}{} & \multicolumn{3}{cV{5}}{\(S\)} & \multicolumn{1}{cV{5}}{\multirow{2}{*}{\(|V|\)}} & \multicolumn{3}{cV{5}}{\(S\)} & \multicolumn{1}{cV{5}}{\multirow{2}{*}{\(|V|\)}} & \multicolumn{3}{cV{5}}{\(S\)} & \multicolumn{1}{c|}{\multirow{2}{*}{\(|V|\)}}\\
    \cline{3-5}\cline{7-9}\cline{11-13}
    \multicolumn{2}{|cV{5}}{} & \(\{0\}\) & \(\{1\}\) & \(\{2\}\) &  & \(\{0\}\) & \(\{1\}\) & \(\{2\}\) &  & \(\{0\}\) & \(\{1\}\) & \(\{2\}\) & \\
    \hlineB{5}
    \multirow{6}{*}{\(n\)} & 6 & \cellcolor{new}\hyperref[thm:k2nk]{NDS*} & ? & \cellcolor{spence}\hyperref[spence]{NDS\(^1\)} & 1395 & ? & ? & \cellcolor{spence}\hyperref[spence]{NDS\(^1\)} & 33 880 & ? & ? & \cellcolor{spence}\hyperref[spence]{NDS\(^1\)} & 376 805\\
    \cline{2-14}
    & 7 & \cellcolor{new}\hyperref[thm:k2nk]{NDS*} & ? & \cellcolor{spence}\hyperref[spence]{NDS\(^1\)} & 11 811 & ? & ? & \cellcolor{spence}\hyperref[spence]{NDS\(^1\)} & \(\approx10^6\) & ? & ? & \cellcolor{spence}\hyperref[spence]{NDS\(^1\)} & \(\approx10^7\)\\
    \cline{2-14}
    & 8 & \cellcolor{new}\hyperref[thm:k2nk]{NDS*} & ? & \cellcolor{spence}\hyperref[spence]{NDS\(^1\)} & 97 155 & ? & ? & \cellcolor{spence}\hyperref[spence]{NDS\(^1\)} & \(\approx10^7\) & ? & ? & \cellcolor{spence}\hyperref[spence]{NDS\(^1\)} & \(\approx10^9\)\\
    \cline{2-14}
    & 9 & \cellcolor{new}\hyperref[thm:k2nk]{NDS*} & ? & \cellcolor{spence}\hyperref[spence]{NDS\(^1\)} & \(\approx10^6\) & ? & ? & \cellcolor{spence}\hyperref[spence]{NDS\(^1\)} & \(\approx10^9\) & ? & ? & \cellcolor{spence}\hyperref[spence]{NDS\(^1\)} & \(\approx10^{11}\)\\
    \cline{2-14}
    & 10 & \cellcolor{new}\hyperref[thm:k2nk]{NDS*} & ? & \cellcolor{spence}\hyperref[spence]{NDS\(^1\)} & \(\approx10^7\) & ? & ? & \cellcolor{spence}\hyperref[spence]{NDS\(^1\)} & \(\approx10^{10}\) & ? & ? & \cellcolor{spence}\hyperref[spence]{NDS\(^1\)} & \(\approx10^{13}\)\\
    \cline{2-14}
    & 11 & \cellcolor{new}\hyperref[thm:k2nk]{NDS*} & ? & \cellcolor{spence}\hyperref[spence]{NDS\(^1\)} & \(\approx10^8\) & ? & ? & \cellcolor{spence}\hyperref[spence]{NDS\(^1\)} & \(\approx10^{11}\) & ? & ? & \cellcolor{spence}\hyperref[spence]{NDS\(^1\)} & \(\approx10^{14}\)\\
    \hline
    \end{tabular}
    \caption{Cospectrality of small generalized Grassmann graphs with \(k=3\) and \(q=2,3,4\).}
    \label{tab:grassmann3}
\end{table}
    \vfill
    \begin{table}[H]
    \centering
    \scalebox{0.75}{
    \begin{tabular}{|c|cV{5}c|c|c|c|c|c|cV{5}cV{5}c|c|c|c|c|c|cV{5}cV{5}c|c|c|c|c|c|cV{5}c|}
    \hline
    \multicolumn{2}{|cV{5}}{\multirow{3}{*}{\({J_{q,S}(n,4)}\)}} & \multicolumn{8}{cV{5}}{\(q=2\)} & \multicolumn{8}{cV{5}}{\(q=3\)} & \multicolumn{8}{c|}{\(q=4\)}\\
    \cline{3-26}
    \multicolumn{2}{|cV{5}}{} & \multicolumn{7}{cV{5}}{\(S\)} & \multicolumn{1}{cV{5}}{\multirow{2}{*}{\(|V|\)}} & \multicolumn{7}{cV{5}}{\(S\)} & \multicolumn{1}{cV{5}}{\multirow{2}{*}{\(|V|\)}} & \multicolumn{7}{cV{5}}{\(S\)} & \multicolumn{1}{c|}{\multirow{2}{*}{\(|V|\)}}\\
    \cline{3-9}\cline{11-17}\cline{19-25}
    \multicolumn{2}{|cV{5}}{} & \(\{0\}\) & \(\{1\}\) & \(\{2\}\) & \(\{3\}\) & \(\{0,1\}\) & \(\{0,2\}\) & \(\{0,3\}\) & & \(\{0\}\) & \(\{1\}\) & \(\{2\}\) & \(\{3\}\) & \(\{0,1\}\) & \(\{0,2\}\) & \(\{0,3\}\) & & \(\{0\}\) & \(\{1\}\) & \(\{2\}\) & \(\{3\}\) & \(\{0,1\}\) & \(\{0,2\}\) & \(\{0,3\}\) & \\
    \hlineB{5}
    \multirow{6}{*}{\(n\)} & 8 & \cellcolor{new}\hyperref[thm:k2nk]{NDS*} & ? & ? & \cellcolor{spence}\hyperref[spence]{NDS\(^1\)} & ? & ? & ? & 200 787 & ? & ? & ? & \cellcolor{spence}\hyperref[spence]{NDS\(^1\)} & ? & ? & ? & \(\approx10^8\) & ? & ? & ? & \cellcolor{spence}\hyperref[spence]{NDS\(^1\)} & ? & ? & ? & \(\approx10^{10}\)\\
    \cline{2-26}
    & 9 & \cellcolor{new}\hyperref[thm:k2nk]{NDS*} & ? & ? & \cellcolor{spence}\hyperref[spence]{NDS\(^1\)} & ? & ? & ? & \(\approx10^6\) & ? & ? & ? & \cellcolor{spence}\hyperref[spence]{NDS\(^1\)} & ? & ? & ? & \(\approx10^{10}\) & ? & ? & ? & \cellcolor{spence}\hyperref[spence]{NDS\(^1\)} & ? & ? & ? & \(\approx10^{12}\)\\
    \cline{2-26}
    & 10 & \cellcolor{new}\hyperref[thm:k2nk]{NDS*} & ? & ? & \cellcolor{spence}\hyperref[spence]{NDS\(^1\)} & ? & ? & ? & \(\approx10^8\) & ? & ? & ? & \cellcolor{spence}\hyperref[spence]{NDS\(^1\)} & ? & ? & ? & \(\approx10^{12}\) & ? & ? & ? & \cellcolor{spence}\hyperref[spence]{NDS\(^1\)} & ? & ? & ? & \(\approx10^{14}\)\\
    \cline{2-26}
    & 11 & \cellcolor{new}\hyperref[thm:k2nk]{NDS*} & ? & ? & \cellcolor{spence}\hyperref[spence]{NDS\(^1\)} & ? & ? & ? & \(\approx10^9\) & ? & ? & ? & \cellcolor{spence}\hyperref[spence]{NDS\(^1\)} & ? & ? & ? & \(\approx10^{13}\) & ? & ? & ? & \cellcolor{spence}\hyperref[spence]{NDS\(^1\)} & ? & ? & ? & \(\approx10^{17}\)\\
    \cline{2-26}
    & 12 & \cellcolor{new}\hyperref[thm:k2nk]{NDS*} & ? & ? & \cellcolor{spence}\hyperref[spence]{NDS\(^1\)} & ? & ? & ? & \(\approx10^{10}\) & ? & ? & ? & \cellcolor{spence}\hyperref[spence]{NDS\(^1\)} & ? & ? & ? & \(\approx10^{15}\) & ? & ? & ? & \cellcolor{spence}\hyperref[spence]{NDS\(^1\)} & ? & ? & ? & \(\approx10^{19}\)\\
    \cline{2-26}
    & 13 & \cellcolor{new}\hyperref[thm:k2nk]{NDS*} & ? & ? & \cellcolor{spence}\hyperref[spence]{NDS\(^1\)} & ? & ? & ? & \(\approx10^{11}\) & ? & ? & ? & \cellcolor{spence}\hyperref[spence]{NDS\(^1\)} & ? & ? & ? & \(\approx10^{17}\) & ? & ? & ? & \cellcolor{spence}\hyperref[spence]{NDS\(^1\)} & ? & ? & ? & \(\approx10^{22}\)\\
    \hline
    \end{tabular}}
    \caption{Cospectrality of small generalized Grassmann graphs with \(k=4\) and \(q=2,3,4\).}
    \label{tab:grassmann4}
\end{table}
    \vfill
\end{landscape}


\begin{thebibliography}{99}





\bibitem{chang} L. Chang, The uniqueness and nonuniqueness of triangular association schemes, 
{\it Science Record} {\bf 3} (1959), 604--613.

\bibitem{cioba} S.M. Cioab{\u{a}}, W.H. Haemers, T. Johnston and M. McGinnis, Cospectral mates for the union of some classes in the Johnson association scheme,
{\it Linear Algebra and its Applications} {\bf 539} (2018), 219--228.

\bibitem{spence} E.R. van Dam, W.H. Haemers, J.H. Koolen and E. Spence, Characterizing distance-regularity of graphs by the spectrum,
{\it Journal of Combinatorial Theory, Series A} {\bf 113} (2006), 1805--1820.


\bibitem{koolen} E.R. van Dam and J.H. Koolen, A new family of distance-regular graphs with unbounded diameter,
{\it Inventiones Mathematicae} {\bf 162} (2005), 189--193.


\bibitem{gm} C.D. Godsil and B.D. McKay, Constructing cospectral graphs,
{\it Aequationes Mathematicae} {\bf 25} (1982), 257--268.

\bibitem{ramezani} W.H. Haemers and F. Ramezani, Graphs cospectral with Kneser graphs,
{\it Graphs and Combinatorics} {\bf 531} (2010), 159--164.

\bibitem{hoffman} A.J. Hoffman, On the uniqueness of the triangular association scheme,
{\it The Annals of Mathematical Statistics} {\bf 31} (1960), 492--497.

\bibitem{liu} T. Huang and C. Liu, Spectral characterization of some generalized Odd graphs, {\it Graphs and Combinatorics} {\bf 15} (1999), 195--209.

\bibitem{ihringer} F. Ihringer and A. Munemasa, New strongly regular graphs from finite geometries via switching, {\it Linear Algebra and its Applications} {\bf 580} (2019), 464--474.
    

\bibitem{QJW2020} L. Qiu, Y. Ji, W. Wang, On a theorem of Godsil and McKay concerning 
the construction of cospectral graphs, \emph{Linear Algebra and its Applications} 
{\bf 603} (2020), 265--274.

\bibitem{wqh} W. Wang, L. Qiu and Y. Hu, Cospectral graphs, GM-switching and regular rational orthogonal matrices of level $p$,
{\it Linear Algebra and its Applications} {\bf 563} (2019), 154--177.










\end{thebibliography}
\end{document}